\documentclass[12pt,twoside,leqno]{amsart}
\usepackage[centertags]{amsmath}
\usepackage{amssymb,amscd,amsfonts,latexsym,amsthm}

\advance\headheight by 3pt

\renewcommand{\subjclass}[1]{\thanks{\emph{2010 Mathematics Subject Classification:}~#1}}
\renewcommand{\keywords}[1]{\thanks{\emph{Keywords and Phrases:}~#1}}
\renewcommand{\date}{\thanks{\today}}

\newtheorem{theorem}{Theorem}[section]
\newtheorem{lemma}{Lemma}[section]

\newtheorem{proposition}[lemma]{Proposition}
\numberwithin{equation}{section}

\def\teto#1{\setbox\z@\hbox{${#1\vphantom k}$}\hbox{%
 \hbox{\lower2\ex@\hbox{\lower\dp\z@\hbox{\vbox{\hrule
 \hbox{\vrule\hskip2\ex@\vbox{\vskip2\ex@\box\z@\vskip1\ex@}%
 \hskip2\ex@\vrule}}}}}}}
\catcode`\@=\active

\newcommand{\Nn}{\mathbb{N}}

\newcommand{\D}{\mathcal{D}}

\renewcommand{\eqref}[1]{(\ref{#1})}

\title[Finiteness results for Diophantine triples]{Finiteness results for Diophantine triples with repdigit values}

\author[A. B\'erczes]{Attila B\'erczes}

\author[F. Luca]{Florian Luca}

\author[I. Pink]{Istv\'an Pink}

\author[V. Ziegler]{Volker Ziegler}



\address{A. B\'erczes \newline
         \indent Institute of Mathematics, University of Debrecen \newline
         \indent H-4010 Debrecen, P.O. Box 12, Hungary}

\email{berczesa\char'100science.unideb.hu}

\address{I. Pink \newline
         \indent Institute of Mathematics, University of Debrecen \newline
         \indent H-4010 Debrecen, P.O. Box 12, Hungary \newline
         \indent and \newline
         \indent University of Salzburg \newline
         \indent Hellbrunnerstrasse 34/I \newline
         \indent A-5020 Salzburg, Austria}

\email{pinki\char'100science.unideb.hu; istvan.pink\char'100sbg.ac.at}

\address{F. Luca \newline
         \indent  School of Mathematics\newline
         \indent University of the Witwatersrand\newline
         \indent Private Bag X3, Wits 2050, South Africa}

\email{florian.luca\char'100wits.ac.za}

\address{V. Ziegler \newline
         \indent University of Salzburg \newline
         \indent Hellbrunnerstrasse 34/I \newline
         \indent A-5020 Salzburg, Austria}

\email{volker.ziegler\char'100sbg.ac.at}

\begin{document}

\baselineskip=17pt

\begin{abstract}
Let $g\ge 2$ be an integer and $\mathcal R_g\subset \Nn$ be the set of repdigits in base $g$. Let $\mathcal D_g$
 be the set of Diophantine triples with values in $\mathcal R_g$; that is, $\mathcal D_g$ is the set of all
 triples $(a,b,c)\in \Nn^3$ with $c<b<a$ such that $ab+1,ac+1$ and $ab+1$ lie in the set $\mathcal R_g$. In this paper,
 we prove effective finiteness results for the set $\mathcal D_g$.
\end{abstract}

\maketitle

\section{Introduction}\label{S_Int}

A classical Diophantine $m$-tuple is a set of $m$ positive integers $\{a_1,\dots,a_m\}$, such
that $a_ia_j+1$ is a square for all indices $1\leq i<j\leq m$. Dujella \cite{Dujella:2004} proved that there is no Diophantine sextuple and that there are only finitely many
Diophantine quintuples. A folklore conjecture is that there are no Diophantine quintuples.
Various variants of the notion of Diophantine tuples have been considered in which the set of squares has been replaced by some other arithmetically interesting subset of the positive integers. For instance, the case of $k$-th powers was considered in \cite{Bugeaud:2003}, while the case of the members of a fixed binary recurrence was considered in \cite{Fuchs:2008, LSz:2008, LSz:2009}. In \cite{ISz}, it is proved that there is no triple of
positive integers $\{a,b,c\}$ such that all of $ab+1, ~ac+1,~ bc+1$ belong to the sequence $\{u_n\}_{n\ge 0}$ of recurrence $u_n=Au_{n-1}-u_{n-2}$ for $n\ge 2$ and initial values $u_0=0$ and $u_1=1$. For related results, see \cite{AISz, RL_glasnik, RL_publ}. The Diophantine tuples with values in the set of $S$-units for a fixed  finite set of primes $S$ was considered in \cite{Luca:2014, Szalay:2013}. For a survey on this topic, we recommend the Diophantine
$m$-tuples page maintained by A. Dujella \cite{Dujella:HP}.

Here we take an integer $g\ge 2$ and recall that a repdigit $N$ in base $g$ is a positive integer all whose base $g$ digits are the same.  That is
\begin{equation}
\label{eq:repdigit}
N=d\left(\frac{g^k-1}{g-1}\right)\qquad {\text{\rm for~some}}\qquad d\in \{1,2,\ldots,g-1\}.
\end{equation}
These numbers
fascinated both mathematicians and amateurs. Questions concerning Diophantine equations involving repdigits have been
considered by Keith \cite{Keith}, Marques and Togb\'e \cite{MT} and Kov\'acs et.al. \cite{KPV}, to name just a few. In this paper, we
combine the Diophantine tuples with repdigits and thus consider Diophantine triples having products increased by $1$ in the set of repdigits in the fixed base $g$.

To avoid trivialities, we only look at repdigits with at least two digits. That is, the parameter $k$ appearing in \eqref{eq:repdigit} satisfies $k\ge 2$. We denote by $\mathcal R_g$ the set of all positive integers that are repdigits in base $g$.
In this paper, we are interested in triples $(a,b,c)\in \Nn^3$, with $c<b<a$ such that $ab+1,ac+1$ and $ab+1$ are all elements
of $\mathcal R_g$. Let us denote by $\mathcal D_g$ the set of all such triples. The reason why we exclude the one-digit
numbers from our analysis is, that in some sense, these are degenerate examples. Furthermore, if we allow $ab+1,ac+1$ and $bc+1$ to be one-digit numbers in
a large base $g$, we will then have many small examples, which however are of no  interest.

Our main result is the following.

\begin{theorem}\label{T_Main}
Assume that $(a,b,c)\in \D_g$. Then
$$a \leq \frac{g^{186}-2}{2}$$
for all integers $g\geq 2$ and
$$a \leq \frac{g^{124}-2}{2}$$
for all integers $g\geq 10^6$.
Moreover, we have
$$\sharp \D_g \leq
\frac{(185g-185)(185g-186)(185g-187)}6.
$$
for all bases $g$ and
\begin{equation}\label{eq:count}
\#{\mathcal D_g}\ll g^{1+o(1)}\qquad {\text{ as}}\qquad g\to \infty.
\end{equation}
\end{theorem}

In the next section, we estimate the greatest common divisor of two numbers of a special shape, which is an important  step in the proof of Theorem \ref{T_Main}.
In Section~\ref{Sec:proof_Th}, we prove Theorem \ref{T_Main} except for the asymptotic bound \eqref{eq:count}, which is proved later in
in Section~\ref{S:count}.

We want to emphasize that our proof of Theorem \ref{T_Main} yields a rather efficient algorithm to compute
 $\D_g$ for a given $g$. In particular, we have computed all sets $\D_g$ for $2\leq g \leq 200$ and we give the details  and the results of this computation
in the last section.

\section{Estimates for the GCD of some numbers of special shape}

The main result of this section is:

\begin{lemma}\label{L_gcd}
Let $g \ge 2$, $k_1, k_2 \geq 1$, $t_1,w_1,t_2,w_2$ be non-zero  integers, and put $C:=\max\{ g, |t_1|,|w_1|,|t_2|,|w_2|\}$.
Let
$$
\Delta=\gcd(t_1g^{k_1}-w_1,t_2g^{k_2}-w_2)
$$
and let $X$ be any real number with $X\geq \max \{ k_1,k_2,3\}$. If $t_1g^{k_1}/w_1$ and $t_2g^{k_2}/w_2$ are multiplicatively independent,
then we have
$$
\Delta\leq 2C^{2+5{\sqrt{X}}}.
$$

\end{lemma}

The proof of this lemma depends, among others, on the following result whose proof is based on the pigeon-hole principle and appears explicitly in \cite{LZ}.

\begin{lemma}\label{L_Claim1} Let $m, n$ and $X$ be non-negative integers such that not both $m$ and $n$ are zero and
such that $X \ge \max\{3, m, n\}$. Then there exist integers $(u,v) \ne (0,0)$ such that
$$\max\{|u|, |v|\} \le \sqrt{X}\quad \text{and} \quad 0 \le mu + nv \le 2\sqrt{X}.$$
\end{lemma}

For a proof of this lemma, see \cite[Claim 1]{LZ}.

\begin{proof}[Proof of Lemma \ref{L_gcd}]
Put $\lambda_i=\gcd(t_i g^{k_i},w_i)$ for $i=1,2$. We have that
$$
t_i g^{k_i}-w_i=\lambda_i\left(t_ig^{k_i}/\lambda_i-w_i/\lambda_i\right)\qquad (i=1,2).
$$
We then have $\Delta=\lambda_1\lambda_2 \Delta_1$, with
$$
\Delta_1=\gcd(t_1 g^{k_1}/\lambda_1-w_1/\lambda_1,t_2 g^{k_2}/\lambda_2-w_2/\lambda_2).
$$
Since $|\lambda_i|\le |w_i|\leq C$ for $i=1,2$, we get the upper bound
\begin{equation}
\label{eq:DD1}
\Delta\le C^2 \Delta_1.
\end{equation}
Thus, it remains to bound $\Delta_1$.

Now, let us consider the pair of congruences
\begin{equation}
\label{eq:cong}
t_i g^{k_i}/\lambda_i\equiv w_i/\lambda_i\pmod {\Delta_1}\qquad (i=1,2)
\end{equation}
and let us note that $w_i/\lambda_i$ and $t_i g^{k_i}/\lambda_i$ are invertible modulo $\Delta_1$.
Indeed, by equation \eqref{eq:cong} there exists an integer $q$ such that
$$t_i g^{k_i}/\lambda_i-w_i/\lambda_i=q\Delta_1.$$
In case that $w_i/\lambda_i$ and $\Delta_1$ have a common prime factor $p$, then $p\mid t_i g^{k_i}/\lambda_i$, contradicting the fact
that $t_ig^{k_i}/\lambda_i$ and $w_i/\lambda_i$ are coprime.

By Lemma \ref{L_Claim1}, we can find a pair of integers $(u_1,u_2)\ne (0,0)$ such that
$$
\max\{|u_1|, |u_2|\}\le {\sqrt{X}},\qquad {\text{\rm and}}\qquad 0\leq u_1 k_1+u_2k_2 \le 2{\sqrt{X}}.
$$
Since we now know
that both sides of \eqref{eq:cong} are invertible modulo $\Delta_1$, it makes sense to take $u_i$-th powers on both sides of \eqref{eq:cong}
for $i=1,2$. Multiplying the two resulting  congruences, we get
\begin{equation}\label{eq:not_zero}
\frac{t_1^{u_1}  t_2^{u_2}g^{k_1u_1+k_2u_2}}{\lambda_1^{u_1}\lambda_2^{u_2} }-
\frac{w_1^{u_1} w_2^{u_2}}{ \lambda_1^{u_1}\lambda_2^{u_2}}\equiv 0\pmod {\Delta_1}.
\end{equation}
The rational number on the left--hand side of \eqref{eq:not_zero} is non-zero, since otherwise we would get
$$
\left(\frac{t_1 g^{k_1}}{w_1}\right)^{u_1} \left(\frac{t_2 g^{k_2}}{w_2}\right)^{u_2}=1,
$$
which implies that $t_1g^{k_1}/w_1$ and $t_2g^{k_2}/w_2$ are multiplicatively dependent because $(u_1,u_2)\ne (0,0)$.
But this is excluded by our hypothesis. Thus, the left hand-side of \eqref{eq:not_zero} is a
non-zero rational number whose numerator is divisible by $\Delta_1$.

Therefore we can write
\begin{equation}
\label{eq:ABC}
\frac{t_1^{u_1} t_2^{u_2} g^{k_1u_1+k_2u_2}}{\lambda_1^{u_1} \lambda_2^{u_2}}=\frac{A B_1B_2}{C_1C_2},
\end{equation}
where $A=g^{k_1u_1+k_2u_2}$ and $\{B_1,B_2,C_1,C_2\}=\{t_1^{|u_1|}, t_2^{|u_2|}, \lambda_1^{|u_1|}, \lambda_2^{|u_2|}\}$.
Similarly, we have
$$
\frac{w_1^{u_1} w_2^{u_2}}{\lambda_1^{u_1}\lambda_2^{u_2}}=\frac{D_1D_2}{E_1E_2},
$$
where $\{D_1,D_2,E_1,E_2\}=\{w_1^{|u_1|}, w_2^{|u_2|}, \lambda_1^{|u_1|}, \lambda_2^{|u_2|}\}.$ Clearly, $|A|\le C^{2{\sqrt{X}}}$, whereas
$$
\max_{i=1,2}\{|B_i|,|C_i|,|D_i|,|E_i|\}\le C^{\sqrt{X}}.
$$

First, let us assume that $u_1u_2\geq 0$. Then $u_1$ and $u_2$ have the same sign and
$$
\max\{k_1, k_2\}< k_1|u_1|+k_2|u_2|=|k_1u_1+k_2u_2|\le 2{\sqrt{X}},
$$
which yields
\begin{equation}
\label{eq:option1}
\Delta_1\le \max\{|t_1 g^{k_1}-w_1|,|t_2 g^{k_2}-w_2|\}\le 2C^{1+2{\sqrt{X}}}\leq 2C^{5\sqrt{X}}.
\end{equation}

Next, we assume that $u_1u_2< 0$, which immediately yields that $\{C_1,C_2\}$ and $\{E_1,E_2\}$
have a common element. Without loss of generality, we may assume that $u_1>0$ and $u_2<0$.
Then we can choose $\lambda_1^{u_1}=C_1=E_1$ and $\Delta_1$ divides the numerator of
$$
\frac{AB_1B_2}{C_1 C_2}-\frac{D_1D_2}{C_1E_2}=\frac{AB_1B_2E_2-C_2D_1D_2}{C_1C_2E_2}.
$$
That is, $\Delta_1\mid AB_1B_2E_2-D_1D_2 C_2$. Since $AB_1B_2E_2-D_1D_2 C_2\neq 0$, we obtain that
\begin{equation}
\label{eq:option2}
\Delta_1\le 2C^{5{\sqrt{X}}}.
\end{equation}
Therefore, we conclude by \eqref{eq:option1} and \eqref{eq:option2}, together with \eqref{eq:DD1}, that
$$
\Delta\le 2C^{2+5{\sqrt{X}}}.
$$
\end{proof}

\section{Proof of Theorem \ref{T_Main}}\label{Sec:proof_Th}

Assume that $(a,b,c)\in\D_g$. By the definition of $\D_g$, we have
\begin{eqnarray}
\label{eq:main}
ab+1 & = & d_3\frac{g^{n_3}-1}{g-1},\nonumber\\
ac+1 & = & d_2\frac{g^{n_2}-1}{g-1},\\
bc+1 & = & d_1\frac{g^{n_1}-1}{g-1},\nonumber
\end{eqnarray}
where $d_i\in \{1,\ldots,g-1\}$ and $n_i \ge 2$ for $i=1,2,3$. It is clear that $n_1\le n_2\le n_3$. Further, we may assume
that $g\ge 3$, since if $g=2$, then $d_1=d_2=d_3=1$,
\begin{eqnarray*}
ab & = & 2^{n_3}-2=2(2^{n_3-1}-1),\\
ac & = & 2^{n_2}-2=2(2^{n_2-1}-1),\\
bc & = & 2^{n_1}-2=2(2^{n_1-1}-1),
\end{eqnarray*}
and by multiplying the above equations we get
$$
(abc)^2=8(2^{n_3-1}-1)(2^{n_2-1}-1)(2^{n_1-1}-1),
$$
which yields a contradiction since the left--hand side is a square and the right--hand side
is divisible by $8$ but not by $16$.

Next, we claim that
\begin{equation}
\label{eq:n3}
n_3\le 2n_2.
\end{equation}
In order to prove \eqref{eq:n3}, we note that
$$
a<ac+1\le g^{n_2}-1,
$$
and therefore
$$
g^{n_3-1}+\cdots+1\le \frac{d_3(g^{n_3}-1)}{g-1}=ab+1<a^2<(g^{n_2}-1)^2<g^{2n_2}.
$$
Thus, we have $n_3<2n_2+1$, and \eqref{eq:n3} is proved. Furthermore, let us note that
\begin{equation}
\label{eq:a}
a>(ab+1)^{1/2}\ge (g^{n_3-1}+\cdots+1)^{1/2}>g^{(n_3-1)/2}.
\end{equation}

Let us fix some notations for the rest of this section.
We rewrite the formulas \eqref{eq:main} as:
\begin{eqnarray}
\label{eq:main1}
ab & = & \frac{\lambda_3}{g-1}\left(\frac{d_3 g^{n_3}}{\lambda_3}-\frac{d_3+g-1}{\lambda_3}\right):=\frac{\lambda_3}{g-1}(x_3-y_3),\nonumber\\
ac & = &  \frac{\lambda_2}{g-1}\left(\frac{d_2 g^{n_2}}{\lambda_2}-\frac{d_2+g-1}{\lambda_2}\right):=\frac{\lambda_2}{g-1}(x_2-y_2),\\
bc & = &  \frac{\lambda_1}{g-1}\left(\frac{d_1 g^{n_1}}{\lambda_1}-\frac{d_1+g-1}{\lambda_1}\right):=\frac{\lambda_1}{g-1}(x_1-y_1),\nonumber
\end{eqnarray}
where
$$
\lambda_i=\gcd(d_i g^{n_i},d_i+g-1),\quad x_i=\frac{d_i g^{n_i}}{\lambda_i},\quad y_i=\frac{d_i+g-1}{\lambda_i}\qquad (i=1,2,3).
$$
Note that $\gcd(x_i,y_i)=1$ for $i=1,2,3$. Hence, the fractions $x_i/y_i$ are reduced. Note also that $x_i>y_i$ for $i=1,2,3$.

In order to prove Theorem \ref{T_Main}, we consider several cases.

\medskip

{\bf Case 1.} {\it $x_1/y_1$ and $x_2/y_2$ are multiplicatively dependent {\bf and so are} $x_1/y_1$ and $x_3/y_3$.}

\medskip

In this case all the fractions $x_i/y_i$, with $i=1,2,3$ belong to the same cyclic subgroup of ${\mathbb Q}_+^*$.
Let $\alpha/\beta>1$ be a generator of this subgroup, where $\alpha,\beta$ are coprime integers. Since $x_i/y_i>1$ for $i=1,2,3$,
there exist positive integers $r_i$ for $i=1,2,3$, such that
$$
x_i=\alpha^{r_i}\quad {\text{\rm and}}\quad y_i=\beta^{r_i}\qquad i=1,2,3.
$$
We split this case up into further subcases and start with:

\medskip

{\bf Case 1.1.} {\it Assume that there exist $i\ne j$ such that $r_i=r_j$.}

\medskip

Let us start with the case that $r_3=r_2$.  We then get that
$$
\alpha^{r_3}=\frac{d_3g^{n_3}}{\lambda_3}=\frac{d_2 g^{n_2}}{\lambda_2}=\alpha^{r_2}.
$$
Hence,
$$
g^{n_3-n_2}=\frac{d_2 \lambda_3}{d_3 \lambda_2}.
$$

We claim that $n_3-n_2\in\{0,1\}$. Note that $d_2\le g-1,~\lambda_3\le 2(g-1)$, which yield $d_2\lambda_3\le 2(g-1)^2$.
In case that $d_3\lambda_2\ge 2$, we obtain
$$
g^{n_3-n_2}\le \frac{2(g-1)^2}{2}=(g-1)^2<g^2,
$$
so we have $n_3-n_2\in \{0,1\}$. Therefore, we are left with the case when $d_3 \lambda_2=1$; i.e. $d_3=\lambda_2=1$. But in this
case, we have
$$
\lambda_3=\gcd(d_3g^{n_3},d_3+g-1)=\gcd(g^{n_3},g)=g,
$$
so
$$
g^{n_3-n_2}=\frac{d_2 \lambda_3}{d_3 \lambda_2}=d_2 \lambda_3\le g(g-1)<g^2.
$$
Thus, in all cases we have that $n_3-n_2\in \{0,1\}$.

Let us consider now the case that $n_3-n_2=0$. This means that
\begin{equation}
\label{eq:star}
\frac{d_3}{\lambda_3}=\frac{d_2}{\lambda_2}.
\end{equation}
But we also have
\begin{equation}
\label{eq:star1}
\beta^{r_3}=\frac{d_3+g-1}{\lambda_3}=\frac{d_2+g-1}{\lambda_2}=\beta^{r_2}.
\end{equation}
Combining  \eqref{eq:star} and \eqref{eq:star1}, we obtain that $(g-1)/\lambda_3=(g-1)/\lambda_2$, so $\lambda_2=\lambda_3$.
Now we deduce by \eqref{eq:star} that $d_2=d_3$. Altogether this yields $ab+1=ac+1$, contradicting our assumption that $b>c$.

Now, we consider the case $n_3-n_2=1$. Instead of \eqref{eq:star}, we now have
\begin{equation}
\label{eq:star2}
\frac{d_3 g}{\lambda_3}=\frac{d_2}{\lambda_2}.
\end{equation}
Combining the equations \eqref{eq:star1} and \eqref{eq:star2}, we get
$$
\frac{d_2+g-1}{d_3+g-1}=\frac{\lambda_2}{\lambda_3}=\frac{d_2}{d_3 g},
$$
which leads to
\begin{equation}\label{eq:Case1.1}
d_3 g(d_2+g-1)=d_2(d_3+g-1).
\end{equation}
Assuming that $d_3\ge 2$, equation \eqref{eq:Case1.1} yields
$$
2g^2\leq d_3 g(d_2+g-1)=d_2(d_3+g-1)\leq 2(g-1)^2
$$
a contradiction, so we may assume that $d_3=1$. Inserting $d_3=1$ into \eqref{eq:Case1.1} yields
$$
g(d_2+g-1)=d_2g,
$$
or, equivalently, $g(g-1)=0$, which is obviously false. In particular, we have proved that the case $r_2=r_3$ yields no solution.

The same arguments hold if we replace the quantities $r_3,r_2,n_3,n_2,d_3,d_2$ by $r_2,r_1,n_2,n_1,d_2,d_1$
and $r_3,r_1,n_3,n_1,d_3,d_1$ respectively.
Thus, {\bf Case 1.1.} yields no solution and we assume from now on that $r_1,r_2$ and $r_3$ are pairwise distinct.

\medskip

{\bf Case 1.2.} {\it Assume that $r_3>\max\{r_1,r_2\}$.}

\medskip

With our notations, we have
$$
(g-1)ab=\lambda_3(\alpha^{r_3}-\beta^{r_3})\qquad {\text{\rm and}}\qquad (g-1)ac=\lambda_2(\alpha^{r_2}-\beta^{r_2}),
$$
and obviously $a(g-1)$ is a common divisor of $\lambda_3(\alpha^{r_3}-\beta^{r_3})$ and $\lambda_2(\alpha^{r_2}-\beta^{r_2})$.
Thus, we have
$$
(g-1)a\mid \gcd\left(\lambda_3(\alpha^{r_3}-\beta^{r_3}),\lambda_2(\alpha^{r_2}-\beta^{r_2})\right).
$$
Taking a closer look at the greatest common divisor on the right--hand side above, we  obtain
$$
(g-1)a\mid \lambda_2\lambda_3 (\alpha^{r}-\beta^r),
$$
where $r=\gcd(r_3,r_2)$.
Similarly, we obtain that
$$
(g-1)b\mid \lambda_3\lambda_1 (\alpha^{s}-\beta^s),
$$
where $s=\gcd(r_3,r_1)$. Together, the last two inequalities give
$$
(g-1)^2ab<\lambda_1\lambda_2\lambda_3^2 \alpha^{r+s}.
$$

Let us write $r=r_3/\delta$ and $s=r_3/\lambda$ for some divisors $\delta>1$ and $\lambda>1$ of $r_3$. Note that we cannot have
$\delta=\lambda=2$. Indeed  $\delta=\lambda=2$ yields $r_2=r_1=r_3/2$, which was excluded by {\bf Case 1.1}.
Thus,
$$
ab<\frac{\lambda_1\lambda_2 \lambda_3^2 \alpha^{r+s}}{(g-1)^2}\le 16(g-1)^2 \alpha^{r+s},
$$
and therefore
$$
ab< 16(g-1)^2\alpha^{r+s}=16(g-1)^2 \alpha^{r_3(1/\delta+1/\lambda)}\le 16(g-1)^2\alpha^{5r_3/6}.
$$

On the other hand, we have
$$
(g-1)ab=\lambda_3(\alpha^{r_3}-\beta^{r_3})\ge \alpha^{r_3}-\beta^{r_3}\ge \alpha^{r_3}-2(g-1),
$$
where we used that $\beta^{r_3}=(d_3+g-1)/\lambda_3\le 2(g-1)$. Hence,
$$
\alpha^{r_3}-2(g-1)\le (g-1)ab<16(g-1)^3 \alpha^{5 r_3/6},
$$
and a crude estimate  now yields
$$
\alpha^{r_3}< 16(g-1)^3 \alpha^{5r_3/6}+2(g-1)<17(g-1)^3 \alpha^{5 r_3/6}.
$$
Thus, we have
$$
\alpha^{r_3}<17^6 (g-1)^{18}.
$$
Now combining the various estimates we obtain
\begin{eqnarray*}
g^{n_3-1} & < & \frac{d_3(g^{n_3}-1)}{g-1}-1=ab=\frac{\lambda_3}{g-1}(\alpha^{r_3}-\beta^{r_3})\\
& < & 2\alpha^{r_3}<2\times 17^6 (g-1)^{18}.
\end{eqnarray*}
Since $g\ge 3$, the above inequality gives $n_3\le 28$ and therefore this case does not yield any solution
with $n_3\geq 29$.

\medskip

{\bf Case 1.3. Assume that $r_3<\max\{r_1,r_2\}$.}

\medskip

Let us assume for the moment that $r_3<r_2$. We then get that
\begin{equation}\label{eq:prep_gcd}
(g-1)ab=\lambda_3(\alpha^{r_3}-\beta^{r_3})\qquad {\text{\rm and}} \qquad (g-1)ac=\lambda_2(\alpha^{r_2}-\beta^{r_2}).
\end{equation}
Let us write $\gcd(r_2,r_3)=r_2/\delta$ with some integer $\delta>1$. Then, as before, we get
$$
(g-1)a\le \lambda_2 \lambda_3 (\alpha^{r_2/\delta}-\beta^{r_2/\delta})
$$
and by the second equation \eqref{eq:prep_gcd}, we get
\begin{equation}
\label{eq:lower}
c\ge \frac{\lambda_2(\alpha^{r_2}-\beta^{r_2})}{\lambda_2\lambda_3(\alpha^{r_2/\delta}-\beta^{r_2/\delta})}>\frac{\alpha^{r_2(\delta-1)/\delta}}{2(g-1)}.
\end{equation}
The above bound yields
\begin{equation}
\label{eq:10}
2\alpha^{r_3}\ge \frac{\lambda_3}{g-1}(\alpha^{r_3}-\beta^{r_3})=ab>c^2>\frac{\alpha^{2r_2(\delta-1)/\delta}}{4(g-1)^2}.
\end{equation}
If we assume that $\delta\ge 3$ and since we have $r_2>r_3$, we get
$$2r_2(\delta-1)/\delta>4r_3/3.$$
If we assume that $\delta=2$, then
$$2r_2(\delta-1)/\delta=r_2=2r_3>4r_3/3.$$
In both cases inequality \eqref{eq:10} implies
$$
\alpha^{r_3/3}<8(g-1)^2.
$$
Hence,
$$
g^{n_3-1}<\frac{d_3(g^{n_3}-1)}{g-1}-1=ab=\frac{\lambda_3}{g-1}(\alpha^{r_3}-\beta^{r_3})<2\alpha^{r_3}<2^{10} (g-1)^6,
$$
which has no solution for $n_3\ge 12$ and $g\ge 3$.

The case when $r_1>r_3$ can be dealt with similarly. In particular, we obtain
instead of \eqref{eq:lower} the inequality
$$
b\ge \frac{\alpha^{r_1(\delta-1)/\delta}}{2(g-1)},
$$
where $r_1/\delta=\gcd(r_1,r_3)$. Using the inequality $ab>b^2$ instead of $ab>c^2$ in the middle of
\eqref{eq:10}, we obtain the same bound for $n_3$.

\medskip

{\bf Case 2.} {\it $x_3/y_3$ and $x_2/y_2$ are multiplicatively independent.}

\medskip

By \eqref{eq:main}, we have
\begin{eqnarray*}
(g-1)ab & = &d_3g^{n_3}-(d_3+g-1) \qquad {\text{\rm and}} \\
(g-1)ac & =& d_2g^{n_2}-(d_2+g-1).
\end{eqnarray*}
Hence, we get an upper bound for $a$, namely
\begin{equation}\label{Case2_2}
(g-1)a \le \gcd(d_3g^{n_3}-(d_3+g-1),d_2g^{n_2}-(d_2+g-1)).
\end{equation}
Since, by assumption,
$$
\frac{x_3}{y_3}=\frac{d_3g^{n_3}}{d_3+g-1}\qquad {\text{\rm and}}\qquad \frac{x_2}{y_2}=\frac{d_2g^{n_2}}{d_2+g-1}
$$
are multiplicatively independent, we may apply Lemma \ref{L_gcd} with the parameters
$$
(t_1,w_1,t_2,w_2,k_1,k_2)=(d_3,d_3+g-1,d_2,d_2+g-1,n_3,n_2),
$$
where
$$
\max\{ |t_1|,|w_1|,|t_2|,|w_2|\} \le 2(g-1)\qquad {\text{\rm and}}\qquad \max\{k_1,k_2,3\} \le n_3.
$$
Thus, by Lemma \ref{L_gcd} and \eqref{Case2_2} for $a$, we get
the upper bound
\begin{equation} \label{Case2_3}
a \le 4(2g-2)^{5\sqrt{n_3}+1}.
\end{equation}
On the other hand, we have an upper bound for $n_3$ given by \eqref{eq:a}, namely
\begin{equation} \label{Case2_4}
n_3<\frac{2\log{a}}{\log{g}}+1.
\end{equation}
Combining the inequalities \eqref{Case2_3} and \eqref{Case2_4}, we obtain
\begin{equation} \label{Case2_5}
n_3<\frac{(10\sqrt{n_3}+2)\log(2g-2)+\log{16}}{\log{g}}+1.
\end{equation}
From \eqref{Case2_5}, we get
\begin{equation}
n_3 \le 178,
\end{equation}
which actually occurrs when $g=4$. Note, that if $g=3$ then \eqref{Case2_5} yields $n_3 \le 171$, while for larger values of $g$ we obtain
better upper bounds for $n_3$. In particular, we have $n_3\leq 105$ provided $g$ is large enough.
If we only assume that $g\geq 200$ and $g\geq 10^6$ we obtain that $n_3\leq 135$ and $n_3\leq 116$, respectively.

\medskip

{\bf Case 3.} {\it $x_3/y_3$ and $x_2/y_2$ are multiplicatively dependent and $x_3/y_3$ {\bf and} $x_1/y_1$ are not.}

\medskip

As in {\bf Case 1}, we may write
$$
x_3=\alpha^{r_3}, y_3=\alpha^{r_3}\qquad {\text{\rm and}}\qquad x_2=\alpha^{r_2}, y_2=\alpha^{r_2}.
$$
Let us note that in the
proof of {\bf Case 1.3} we never used the quantity $r_1$ when we considered the case $r_2<r_3$.
Therefore, we may assume $r_3>r_2$.

Similarly as in {\bf Case 2}, we find an upper bound for $b$, but we use
\begin{equation*}
\begin{split}
(g-1)ab=&\,d_3g^{n_3}-(d_3+g-1) \qquad \textrm{and}\\
(g-1)bc=&\,d_2g^{n_1}-(d_1+g-1),
\end{split}
\end{equation*}
instead. Therefore, Lemma \ref{L_gcd}, we obtain the upper bound
\begin{equation} \label{Case3_3}
b \le 4(2g-2)^{5\sqrt{n_3}+1}.
\end{equation}

Next we want to find an upper bound for $a$. To this end, we consider
\begin{equation}
\label{eq:rel}
ab=\frac{\lambda_3}{g-1}(\alpha^{r_3}-\beta^{r_3})\qquad {\text{\rm and}} \qquad ac=\frac{\lambda_2}{g-1}(\alpha^{r_2}-\beta^{r_2}).
\end{equation}
Hence, we obtain
$$
(g-1)a\mid \lambda_3\lambda_2\gcd(\alpha^{r_3}-\beta^{r_3},\alpha^{r_2}-\beta^{r_2}) < 4(g-1)^2 \alpha^r,
$$
where $r=\gcd(r_2,r_3)$. Thus,
\begin{equation}
\label{eq:111}
a< 4(g-1)\alpha^r.
\end{equation}

On the other hand, we have
$$ab+1=\frac{d_3(g^{n_3}-1)}{g-1} \ge g^{n_3-1}+g^{n_3-2}+\cdots+1,$$
that is $a \ge g^{n_3-1}/{b}$, whence by \eqref{Case3_3}, we get
\begin{equation} \label{Case3_5}
a \ge \frac{g^{n_3-1}}{4(2g-2)^{5\sqrt{n_3}+1}}.
\end{equation}
By using \eqref{Case3_3} and the fact that $d_3 \ge 1$ and $d_2 \le g-1$, we  find the following lower bound for $b$:
$$b \ge \frac{b}{c}=\frac{ab}{ac}>\frac{ab+1}{ac+1}=\frac{d_3(g^{n_3}-1)}{d_2(g^{n_2}-1)} \ge \frac{g^{n_3-n_2}}{g-1}, $$
which yields
\begin{equation} \label{Case3_6}
g^{n_3-n_2} < (g-1)b.
\end{equation}

Let us recall that
\begin{equation*}
\alpha^{r_2}=\frac{d_2g^{n_2}}{\lambda_2} \ge \frac{g^{n_2}}{2(g-1)} \;\;
\text{and} \;\; \alpha^{r_3}=\frac{d_3g^{n_3}}{\lambda_3} \le (g-1)g^{n_3}.
\end{equation*}
As in {\bf Case 1}, let $r=\gcd(r_2,r_3)$. We then find
\begin{equation}\label{Case3_9}
\alpha^r \le \alpha^{r_3-r_2}\leq 2(g-1)^2 g^{n_3-n_2}< 2(g-1)^3 b,
\end{equation}
where the last inequality is due to \eqref{Case3_6}.
We combine the inequalities \eqref{Case3_3}, \eqref{eq:111}, \eqref{Case3_5} and \eqref{Case3_9} and obtain
\begin{eqnarray*}
 \frac{g^{n_3-1}}{4(2g-2)^{5\sqrt{n_3}+1}} & \leq & \frac{g^{n_3-1}}{b} \leq a   <  4(g-1)\alpha^r \\
 & < & 8(g-1)^4 b \leq 32(g-1)^4(2g-2)^{5\sqrt{n_3}+1}.
\end{eqnarray*}
Taking logarithms we obtain a similar inequality for $n_3$ as in {\bf Case 2}:
\begin{equation}
n_3<\frac{(10\sqrt{n_3}+2)\log(2g-2)+4\log(g-1)+\log{128}}{\log{g}}+1.
\end{equation}
The above yields
\begin{equation} \label{upper_n3_case3}
n_3 \le 186.
\end{equation}
Note that we obtain $n_3 \le 186$ if $g \in \{4,5\}$, whereas in all other cases we obtain better bounds.
In particular, if we assume that $g\geq 200$, then we obtain that $n_3 \le 143$ and if we assume that $g\geq 10^6$,
then we obtain that $n_3 \le 124$.
Finally, let us note that if $g$ is large enough, then we may even assume that $n_3\leq 113$.

Let us summarize our results so far:

\begin{proposition}\label{Prop:bound_n3}
Assume equations \eqref{eq:main}. We then  have $n_3\leq 186$. If we assume that $g\geq 200$ or that $g\geq 10^6$,
then we have that $n_3 \le 143$ and $n_3 \le 124$, respectively.
Moreover, we even may assume that $n_3\leq 113$, if $g$ is large enough ($g>10^{153}$).
\end{proposition}

Now a simple combinatorial argument concludes the proof of the first part of our theorem. Indeed, the distinct tuples
$(n_1,d_1), (n_2,d_2), (n_3,d_3)$ may be selected from a set of cardinality $185(g-1)$ and altogether in
$$(185g-185)(185g-186)(185g-187)$$
ways. Since only those results are acceptable, where
$$
d_1\frac{g^{n_1}-1}{g-1}<d_2\frac{g^{n_2}-1}{g-1}<d_3\frac{g^{n_3}-1}{g-1}
$$
we are left with
$$
\frac{(185g-185)(185g-186)(185g-187)}{6}$$
possibilities for the tuple $(d_1,n_1,d_2,n_2,d_3,n_3)$.
Further, for a given sextuple $(d_1,n_1,d_2,n_2,d_3,n_3)$, the system of equations
\eqref{eq:main} has at most one solution in positive integers $(a,b,c)$. Additionally, since $b \ge 2$, $d_3 \le g-1$, $n_3\leq 186$
and \eqref{eq:main}, the estimate for $a$ is trivial.
This concludes the proof of the first part of Theorem \ref{T_Main}.

\section{Counting the number of triples}\label{S:count}

We are left with the proof of the last statement of our main Theorem \ref{T_Main}. The main purpose of this section
is to prove Theorem \ref{thm:count} below.
Let $\tilde{\mathcal R_g}$ be the set of repdigits together with the integers of digit length $1$ in base $g$. Denote
by $\tilde{\mathcal D_g}$ the set of triples $(a,b,c)\in {\mathbb N}^3$ such that $1\leq c<b<a$ and $ab+1,ac+1$ and $bc+1$ are elements
of $\tilde{\mathcal R_g}$. We prove the following theorem:

\begin{theorem}
\label{thm:count}
We have
$$
\#\tilde {\mathcal D_g}\asymp g^{3/2}\qquad (g\to \infty),
$$
and
$$
\#{{\mathcal D_g}}\ll g^{1+o(1)}\qquad (g\to \infty).
$$
\end{theorem}

Since $g$ is fixed throughout this section, we will omit the index of $\mathcal D_g$ and $\tilde{\mathcal D}_g$ and write only
$\mathcal D$ and $\tilde{\mathcal D}$, respectively. During the
course of the proof of Theorem \ref{thm:count}, we consider several subsets of $\tilde{\mathcal D}$ which will be denoted by
$\mathcal D_1,\dots ,\mathcal D_4$. We want to emphasize here that in the following a subscript of $\mathcal D$ does not refer to the
base $g$, but instead to a certain subset of $\tilde{\mathcal D}$.

\begin{proof}
Clearly, $\tilde{\mathcal D}$ can also be identified with the set of all sextuples
$$
(d_1,d_2,d_3,n_1,n_2,n_3)\qquad {\text{\rm where}}\qquad 1\leq d_i \leq g-1\quad {\text{\rm for}}\quad i=1,2,3,
$$
such that there exist positive integers $c<b<a$ having the property that equations \eqref{eq:main} hold.
Under this identification and using Proposition~\ref{Prop:bound_n3}, for $g$ large enough we have $n_1\le n_2\le n_3\le 113$ and $1\leq d_i\leq g-1$
for $i=1,2,3$. So, trivially, $\#\tilde{\mathcal D}\ll g^3$.
Let us improve this trivial bound. Let ${\mathcal D}_1$ be the subset of $\tilde{\mathcal D}$ such that $n_3=1$,
and ${\mathcal D}_2=\tilde{\mathcal D}\backslash {\mathcal D}_1$.
We prove:
\begin{itemize}
\item[(i)] $\#{\mathcal D}_1 \asymp g^{3/2}$ as $g\to\infty$.
\item[(ii)] $\#{\mathcal D}_2\ll g^{1+o(1)}$ as $g\to\infty$.
\end{itemize}
The conclusion of Theorem  \ref{thm:count} follows from (i), (ii) and the fact that
$$
\#\tilde{\mathcal D}= \#{\mathcal D}_1+\#{\mathcal D}_2.
$$

First, let us deal with (i). For the lower bound, we choose $a>b>c$ all three in $\{1,2,\ldots,\lfloor {\sqrt{g-2}}\rfloor\}$. For each of these choices,
$$
ab+1\le \lfloor {\sqrt{g-2}}\rfloor^2+1\le g-1,
$$
so $ab+1=d_1\in [1,g-1]$ and similarly $ac+1=d_2,~bc+1=d_3$. Thus, $(a,b,c)$ is in ${\mathcal D}_1$, and we get
\begin{equation}
\label{eq:1}
\#{\mathcal D}_1\ge \binom{\lfloor {\sqrt{g-2}}\rfloor}{3}\gg g^{3/2}.
\end{equation}
For the upper bound, note that we have to count the integers $a>b>c$ satisfying \eqref{eq:main} with $n_1=n_2=n_3=1$. In particular,
we have to count the triples $(a,b,c)$ satisfying
\begin{equation}\label{eq:D1_count}
1\leq a \leq g-2\;\; \text{and}\;\; 1\leq c<b<\min\left\{a,\frac{g-1}{a}\right\}.
\end{equation}
For fixed $a$ there are $\ll\min\{a,g/a\}^2$ pairs $(b,c)$ satisfying \eqref{eq:D1_count}.
Therefore we obtain
$$
\# \mathcal D_1 \ll \int_1^g \min\{a,g/a\}^2 da= \int_1^{\sqrt g} a^2 da+\int_{\sqrt{g}}^{g} \left(\frac ga\right)^2 da \ll g^{3/2},
$$
which is the desired upper bound.

For (ii), let ${\mathcal D}_3$ be the subset of ${\mathcal D}_2$ such that $n_3\ge 3$.
Due to Proposition~\ref{Prop:bound_n3}, for $g$ large enough we may assume that $n_3\le 113$ and $d_3\le g-1$. We look at
\begin{equation}
\label{eq:3}
ab=d_3\left(\frac{g^{n_3}-1}{g-1}\right)-1.
\end{equation}
Clearly, since $n_3\ge 3$, we have $a^2>ab\ge (g^2+g+1)-1>g^2$, so $a>g$.
Since $d_3$ and $g$ are fixed and $n_3\le 113$, the number of ways of choosing $(a,b)$ such that $a>b$ and \eqref{eq:3} holds
is
$$
\tau\left(\frac{d_3(g^{n_3}-1)}{g-1}-1\right)\ll g^{o(1)}\qquad {\text{\rm as}}\qquad g\to\infty,
$$
where $\tau(n)$ is the number of divisors of $n$. The asymptotic bound on the right side follows from a well--known upper bound
for the divisor function (e.g. see \cite[Theorem 2.11]{MV:book} or \cite[Chapter 7.4]{DL:book})
It remains to find out in how many ways we can choose $c$. Well, let us also fix $n_2\le n_3$. Then $d_2\in \{1,\dots,g-1\}$ is such that
$$
d_2\left(\frac{g^{n_2}-1}{g-1}\right)\equiv 1\pmod a.
$$
This puts $d_2$ into a fixed arithmetic progression $\alpha_{n_2}$ modulo $a$, where $\alpha_{n_2}$
is the inverse of $(g^{n_2}-1)/(g-1)$ modulo $a$. We show that this progression contains at most one
value for $d_2$. Assuming this is not the case, let $d_2$ and $d_2'$ be both congruent to $\alpha_{n_2}$ and in the intervall $[1,g-1]$.
Assume that $d_2<d_2'$, then $a\mid d_2'-d_2$, so
$$
g<a\le d_2'-d_2\le g-2,
$$
which is false. This shows that indeed once $d_3,~n_3$ and $a$ (hence also $b$) are determined, then any choice of $n_2\le n_3$
determines $d_2$ (hence, $c$) uniquely. Thus,
\begin{equation}
\label{eq:D3}
\#{\mathcal D_3} \le \sum_{d_3=1}^{g-1} \sum_{n_3=3}^{113} \sum_{n_2=1}^{n_3} \tau_2\left(\frac{d_3(g^{n_3}-1)}{g-1}-1\right)\ll
g^{1+o(1)}\qquad (g\to\infty).
\end{equation}

It remains to find an upper bound for the cardinality of ${\mathcal D}_4:={\mathcal D}_2\backslash {\mathcal D}_3$.
These triples are the ones that have $n_3=2$. We fix $d_3$ and write
\begin{equation}
\label{eq:ab}
ab=d_3(g+1)-1.
\end{equation}
There are at most $\tau_2(d_3(g+1)-1)=g^{o(1)}$ possibilities for $a>b$ satisfying the above relation \eqref{eq:ab}
as $g\to \infty$. It remains to determine the number of choices for $c$. Let us also fix $n_2\le n_3=2$.
Then determining $c$ is equivalent to determining the number of choices for $d_2$ such that
\begin{equation}
\label{eq:4}
d_2\left(\frac{g^{n_2}-1}{g-1}\right)\equiv 1\pmod a,\qquad 1\leq d_2 \leq g-1.
\end{equation}
Congruence \eqref{eq:4} puts $d_2$ in a certain fixed arithmetic progression modulo $a$ and the number of
such numbers $1\leq d_2 \leq g-1$ is at most
$$
1+\left\lfloor \frac{g-1}{a}\right\rfloor.
$$
We assume that $a\le g-1$, otherwise there is at most one choice for $d_2$, and the counting function of such examples is at most $g^{1+o(1)}$ by the argument for $\#{\mathcal D}_3$. Then
the number of choices for $c$ is at most
$$
1+\left \lfloor  \frac{g-1}{a}\right\rfloor\le 1+\frac{g-1}{a}<\frac{2g}{a}\le \frac{2g}{{\sqrt{d_3(g+1)-1}}}\ll \frac{\sqrt{g}}{\sqrt{d_3}}.
$$
This shows that
\begin{eqnarray*}
\#{\mathcal D_4} & \ll & \sum_{n_2\le 2} \sum_{d_3=1}^{g-1} \tau_2(d_3(g+1)-1) \frac{\sqrt{g}}{\sqrt{d_3}}\\
& \ll & g^{1/2+o(1)} \sum_{1\le d_3\le g-1} \frac{1}{\sqrt{d_3}}\ll g^{1/2+o(1)} \int_1^{g-1} \frac{dt}{t^{1/2}}\\
& \ll & g^{1/2+o(1)}\left( 2t^{1/2} \Big|_{t=1}^{t=g-1}\right)
\ll g^{1+o(1)}\qquad (g\to\infty).
\end{eqnarray*}
Together with \eqref{eq:D3}, we get
$$
\#{\mathcal D}_2\le \#{\mathcal D}_3+\#{\mathcal D}_4\le g^{1+o(1)}\qquad (g\to\infty),
$$
which is (ii).
\end{proof}

\section{The case of small bases $g$}

We have computed for the bases $2\leq g\leq 200$ all triples $(a,b,c)\in \D_g$. In particular we found the following triples:

$$
\begin{tabular}{|c|c|c|c|}
  \hline
  $g$ & $a$ & $b$ & $c$ \\
  \hline
  23 & 65 & 17 & 7 \\
  \hline
  42 & 136 & 93 & 6 \\
  \hline
  104 & 292 & 187 & 32 \\
  \hline
  171 & 5607 & 619 & 5 \\
  \hline
  190 & 439 & 248 & 67 \\
  \hline
\end{tabular}
$$

In our computations, we considered all values of $2 < g \leq 200$ one by one and we split our work
depending on the size of $a$. If $g\leq 100$, we put $B:=1000$
and for $101\leq g \leq 200$, we put $B:=10000$.

For every $a<B$ we do as follows: for $2\leq b <a$ we check whether $ab+1$ is a repdigit number in base $g$.
If yes, we also check if we can find $c<b$ such that $ab+1,ac+1$ and $bc+1$ are all repdigit numbers in base $g$.

For $a\geq B$ we proceed as follows: We use equations \eqref{eq:main} and \eqref{eq:n3}.
For all integer values of
$2\leq n_2\leq 186$ and all integer values of $n_3$ between $n_2$ and the minimum of 186 and $2n_2$
and for all possible digits $d_2$ and $d_3$, we compute
$$
ab=d_3\frac{g^{n_3}-1}{g-1}-1, \quad\quad\quad ac=d_2\frac{g^{n_2}-1}{g-1}-1.
$$
Since $a\leq \gcd(ab,ac)$, the cases when $\gcd(ab,ac)<B$ are covered by the cases when $a<B$, so
we only have further work to do if
$$
\gcd\left(d_3\frac{g^{n_3}-1}{g-1}-1, d_2\frac{g^{n_2}-1}{g-1}-1\right)\geq B.
$$
In this case, for every integer $2\leq n_1\leq n_2$ and every digit $d_1$, we check whether
$$
\left(d_3\frac{g^{n_3}-1}{g-1}-1\right) \left(d_2\frac{g^{n_2}-1}{g-1}-1\right)\left(d_1\frac{g^{n_1}-1}{g-1}-1\right)
$$
is a square, and if yes, then we check whether the corresponding values of $a,b$ and $c$ are integers.
If yes, then we found a solution.

We implemented the above algorithm in Magma\cite{MAGMA} and the running time was less than 4 days on an Intel(R) Core(TM) 960  3.2GHz processor.

\subsection*{Acknowledgements}
The research was supported in part by the University of Debrecen (A.B.),
by the J\'anos Bolyai Scholarship of the Hungarian Academy of Sciences (A.B.).
and by grants K100339 (A.B.) and NK104208 (A.B.) of the Hungarian National Foundation for Scientific Research.
The research was also granted by the Austrian science found (FWF) under the project P 24801-N26.

\end{document}